\pdfoutput=1
\documentclass[a4paper,english]{jnsao}
\usepackage[utf8]{inputenc}
\usepackage[english]{babel}
\usepackage{graphicx}
\usepackage{float}
\usepackage{wrapfig}
\usepackage{subcaption}
\usepackage[textsize=footnotesize,color=DarkOrange!40]{todonotes}
\usepackage[indLines=false,noEnd=false]{algpseudocodex}
\usepackage[nameinlink,capitalise]{cleveref}
\usepackage{algorithm}
\usepackage{enumitem}

\theoremstyle{definition}
\newtheorem{assumption}[theorem]{Assumption}
\crefname{assumption}{Assumption}{Assumptions}

\numberwithin{algorithm}{section}

\usepackage{tikz,pgfplots,pgfplotstable}
\usepgfplotslibrary{colorbrewer,fillbetween}
\pgfplotsset{compat=newest}
\pgfplotsset{
    fb/.style                 = { color = Set2-A,          line width = 1pt },
    projgrad/.style           = { color = Set2-B,          line width = 1pt, dashed },
    legend style = {
        inner sep = 0pt,
        outer xsep = 5pt,
        outer ysep = 0pt,
        legend cell align = left,
        align = left,
        draw = none,
        fill = none,
    },
}

\newcommand{\cj}[1]{\{ #1\}}
\newcommand{\proxold}[2]{\prox_{#1}(#2)}
\DeclareMathOperator{\prox}{prox}

\newcommand{\df}[2]{#1 ( #2 )}
\newcommand{\nr}[2]{ \| #1 \|_{#2}}
\newcommand{\pd}[2]{ \langle #1,#2 \rangle}
\newcommand{\rea}[1]{\mathbb{R}^{#1}}

\def\grad{\nabla}
\def\norm#1{\|#1\|}
\def\linear{\mathbb{L}}
\def\term#1{\emph{#1}}
\def\iprod#1#2{\langle #1, #2 \rangle}
\def\defeq{:=}
\def\N{\mathbb{N}}
\def\R{\mathbb{R}}
\def\C{\mathbb{C}}
\def\extR{\widebar{\R}}

\DeclareMathOperator{\sublev}{sub}
\DeclareMathOperator{\diam}{diam}
\DeclareMathOperator{\interior}{int}
\DeclareMathOperator{\boundary}{bd}
\DeclareMathOperator{\ccone}{ccone}
\DeclareMathOperator{\Sym}{Sym}
\let\Re\relax\DeclareMathOperator{\Re}{Re}

\newcommand{\freevar}{\,\boldsymbol\cdot\,}
\def\abs#1{|#1|}
\def\polar#1{#1^\circ}
\def\bipolar#1{#1^{\circ\circ}}
\def\ortho{o}
\def\relerr{\rho}

\manuscriptcopyright{}
\manuscriptlicense{}

\title{Multigrid methods for total variation\texorpdfstring{$^\ddagger$}{}}
\shorttitle{Multigrid methods for total variation}

\date{2024-11-18}

\author{%
    Felipe Guerra\thanks{Research Center in Mathematical Modeling and Optimization (MODEMAT), Quito, Ecuador \emph{and} Department of Mathematics, Escuela Politécnica Nacional (EPN), Quito, Ecuador. \email{edison.guerra@epn.edu.ec}}
    \and
    Tuomo Valkonen\thanks{MODEMAT \emph{and} EPN \emph{and} Department of Mathematics and Statistics, University of Helsinki, Finland. \email{tuomo.valkonen@iki.fi}, \orcid{0000-0001-6683-3572}}
}

\acknowledgements{This research has been supported by the EPN internal project PIS-23-04.

\medskip

${\ \ }^\ddagger$ Deanonymised version submitted to SSVM 2025. This version does not incorporate post peer review improvements.}

\begin{document}

\maketitle

\begin{abstract}
    Based on a \emph{nonsmooth coherence condition}, we construct and prove the convergence of a forward-backward splitting method that alternates between steps on a fine and a coarse grid.
    Our focus is an total variation regularised inverse imaging problems, specifically, their dual problems, for which we develop in detail the relevant coarse-grid problems.
    We demonstrate the performance of our method on total variation denoising and magnetic resonance imaging.
\end{abstract}

\section{Introduction}

In this work, we consider composite optimisation problems of the form
\begin{equation}
    \label{eq:intro:problem:fine-grid}
    \min_{x \in X}~ F(x) + G(x),
\end{equation}
where $F$ is convex and smooth, and $G$ is convex but possibly nonsmooth on a Hilbert space $X$.
We want to apply forward-backward splitting \cite{lions1979splitting}
to this problem, while reducing computational effort by occasionally passing to a lower-dimensional problem.
Such multigrid methods make possible the computationally efficient high-precision solution of partial differential equations \cite{briggs2000multigrid}.
A few works \cite{nash2000multigrid,parpas2017multilevel,kornhuber1994monotone,ang2024mgprox} have looked into applying the same principle to large-scale optimisation problems and variational inequalities.
However, none, so far, treat nonsmoothness completely satisfactorily: while \cite{ang2024mgprox} allows $G$ to be nonsmooth, it requires $F$ to be strongly convex in addition to smooth.
Such an assumption is rarely satisfied in problems of practical interest.
In \cite{kornhuber1994monotone} only constrained quadratic problems are considered.
In \cite{parpas2017multilevel}, less assumptions are imposed on the fine-grid problem, however, the coarse-grid problems are required to be smooth.
We want to exploit the inherent nonsmooth properties of the problem on the coarse grid as well.

We focus on imaging with isotropic total variation regularisation, i.e.,
\[
    \min_{y \in Y}~ E(y) + \alpha\norm{\grad_h y}_{2,1},
\]
where $E$ is a data fitting term; $\grad_h \in \linear(Y; X)$ a (discrete) gradient operator; and $\norm{\freevar}_{2,1}$ the sum over a pixelwise $2$-norms.
Since, in the present work, we are limited to forward-backward splitting, and the proximal map of total variation is not prox-simple, i.e., not easily calculated, we have to work with the dual problem
\begin{equation}
    \label{eq:intro:tv-dual}
    \min_{x \in X}~ E^*(-\grad_h^* x) + \delta_{\alpha B_{2,1}}(x).
\end{equation}
To calculate $\grad E^*$ efficiently, indeed, for $E^*$ to be smooth, we will, unfortunately, need $E$ to be strongly convex (but not necessarily smooth).
Typically $E(y)=\frac{1}{2}\norm{Ty-b}^2$ for a forward operator $A$ mapping an image to its measurements, so we will need $T$ to be invertible.
This holds with fully or over-sampled data.

The problem \eqref{eq:intro:tv-dual} is of the form \eqref{eq:intro:problem:fine-grid}.
A significant step in forming a multigrid optimisation method is deciding on a coarse-grid version of the problem. While MGProx \cite{ang2024mgprox} allows any smooth coarse-grid problem, we will follow the approach of \cite{parpas2017multilevel}---where gradient descent for smooth problems was considered---in proposing in \cref{sec:nscc} a \term{nonsmooth coherence condition} that locally determines the coarse-grid problem. For \cref{eq:intro:tv-dual}, this will form coarse-grid constraints more difficult than $\alpha B_{2,1}$.
We analyse the projection to those constraints in \cref{sec:tv} after proposing and proving the convergence of the general method in \cref{sec:fb}.
We finish with denoising and magnetic resonance imaging (MRI) experiments in \cref{sec:numerical}.

This work builds upon the Master's thesis of the first author \cite{felipe-msc} with simplified proofs and coarse problem, and expanded numerics with a faster implementation \cite{multigrid-codes-zenodo}.

\paragraph{Notation}

Let $X$ and $Y$ be Hilberts spaces, $A \subset X$.
We write $\linear(X; Y)$ for the bounded linear operators between $X$ and $Y$, $\delta_A$ for the $\{0,\infty\}$-valued indicator function of $A$, and
$B(x,\alpha)$ for the closed ball of radius $\alpha$ and centre $x$ in $X$.
We write $\polar A \defeq \{ z \mid \iprod{z}{x} \le 0\ \forall x \in A\}$ for the polar, and $\bipolar A \defeq \polar{(\polar A)}$ for the bipolar.
These satisfy $\bipolar A \supset A$ and $\polar{(\bipolar A)} = \polar A$.
The smallest closed convex cone containing $A$ is $\ccone A \defeq \bipolar A$.
The normal cone to $A$ at $x$ is $N_A(x) \defeq \{ z \mid \iprod{z}{\tilde x - x} \le 0 \ \forall \tilde x \in A\}$.
For a convex $F: X \to \extR$, the subdifferential at $x$ is $\partial F(x)$, the proximal operator $\prox_F$, and the Fenchel conjugate $F^*$.
We have $\partial \delta_A(x) = N_A(x)$.
We refer to \cite{clason2020introduction} for more details on these concepts.
Finally $x_{\freevar j} \in \R^D$ is the $j$:th row of $x \in \R^{D \times n}$.

\section{The coarse problem}
\label{sec:nscc}

The first question we have to answer is how to build the coarse problem?
In this section, we construct a general guideline, the \emph{nonsmooth coherence condition}, and prove that a forward-backward method applied to a coarse-grid problem satisfying this condition, will construct a descent direction for the fine grid.

\subsection{The nonsmooth coherence condition}

Write $I_h^H \in \linear(X; X_H)$ for the \term{restriction} operator from the fine grid modelled by the Hilbert space $X$ to the coarse grid modelled by the Hilbert space $X_H$. Typically $X$ and $X_H$ are finite-dimensional with $\dim X_H \ll \dim X$.
We call $I_H^h \in \linear(X_H; X)$ satisfying $\mu I_H^h=(I_h^H)^*$ for some $\mu>0$ the \term{prolongation} operator.

In \cite{nash2000multigrid,parpas2017multilevel}, treating the smooth problem $\min_x F(x)$, the coarse-grid problems $\min_x F_H^k$ were built by introducing an arbitrary coarse objective $F_H$ that satisfies the smooth \term{coherence condition}
$
    I_h^H \nabla F(x^{k}) = \nabla F_H(\zeta ^{k,0}),
$
for an initial coarse point $\zeta^{k,0}$, typically $\zeta^{k,0}=I_h^H x^k \in X_H$, and setting
\[
    F_H^k(\zeta) \defeq F_H(\zeta) + \iprod{w_H^k}{\zeta-\zeta^{k,0}}
    \quad\text{for}\quad
    w_H^k \defeq I_h^H \grad F(x^k) - \grad F_H(\zeta^{k,0}) \in X_H.
\]
Then $\grad F_H^k(\zeta^{k,0})=\grad F(x^k)$, so if $x^k$ solves the fine-grid problem, the coarse grid problem triggers no change: it is solved by $\zeta^{k,0}$.
A descent direction for the coarse-grid problem also allows constructing a find-grid descent direction \cite{nash2000multigrid}.

We extend this approach to the nonsmooth problems \eqref{eq:intro:problem:fine-grid}.
Specifically, for $G_H^k$ satisfying the following two assumptions, we take as our \term{coarse-grid problem}
\begin{align}
    \label{eq:notation:problem:coarse-grid}
    \min_{\zeta \in X_{H}} F_H^k(\zeta) + G_H^k(\zeta).
\end{align}

\begin{assumption}[Basic coarse structure]
    \label{ass:coarse:basic}
    $F_H: X_H \to \R$ is convex with $L_H$-Lipschitz gradient. For all $k \in \N$, $G_H^k: X \to \extR$ is convex, proper, lower semicontinuous.
    The step length parameter $\tau_H>0$ satisfies $\epsilon \defeq 2- \tau_HL_H > 0$.
\end{assumption}

\begin{assumption}[Nonsmooth coherence condition]
    \label{ass:coarse:coherence}
    The fine-grid iterate $x^k$ and the initial coarse iterate $\zeta^{k,0}$
    (typically $I_h^H x^k$) satisfy
    $
        \df{I_{h}^{H}\partial G}{x^{k}}\subseteq \df{\partial G_{H}^{k}}{\zeta^{k,0}}.
    $
\end{assumption}

\begin{example}
    \label{ex:coarse:canonical:gH}
    Take $G_H^k=\delta_{\Omega^k}$ for $\Omega^k \defeq  \zeta^{k,0} + \polar{(I_h^H\partial G(x^k))}$.
    Obviously, $\Omega^k \subset X_H$ is nonempty and convex, and $\partial G_H^k(\zeta^{k,0})=N_{\Omega_k}(\zeta^{k,0})=\bipolar{(I_h^H\partial G(x^k))} \supset I_h^H\partial G(x^k)$.
\end{example}

\subsection{Coarse-grid algorithm and descent directions}

We apply forward-backward splitting to the coarse-grid problems \eqref{eq:notation:problem:coarse-grid}.
For an initial coarse point $\zeta^{k,0}$ and an iteration count $m \in \N$, we  thus iterate
\begin{equation}
    \label{eq:coarse:alg}
    \zeta^{k,j+1}
    \defeq
    \prox_{\tau_H G_H^k}(\zeta^{k,j} - \tau_H \grad F_H^k(\zeta^{k,j}))
    \quad
    (j=0,\ldots,m-1).
\end{equation}
In the following we show that $d=\zeta^{k,m}-\zeta^{k,0}$ is a fine-grid descent direction.

\begin{lemma}
    \label{lemma:coarse:descent}
    If \cref{ass:coarse:basic} holds, and we apply \eqref{eq:coarse:alg} for any $\zeta ^{k,0} \in X_H$, then
    \[
        \pd{I_h^H \nabla F(x^k)}{\zeta^{k,m}-\zeta^{k,0}}+\frac{\epsilon}{2\tau_H}\sum _{j=0}^{m-1}\nr{\zeta^{k,j+1}-\zeta^{k,j}}{X_H}^{2}\leq \df{G_{H}^{k}}{\zeta^{k,0}} - \df{G_{H}^{k}}{\zeta^{k,m}}.
    \]
\end{lemma}

\begin{proof}
    We abbreviate $J_H^k \defeq G_H^k + F_H$ and $\zeta^j \defeq \zeta^{k,j}$.
    In implicit form, \eqref{eq:coarse:alg} reads
    \begin{align}
        \label{eq:coarse:ei}
        0\in \df{\partial G_{H}^k}{\zeta^{j+1}} + \df{\nabla F_{H}}{\zeta ^{k,j}}+w_{H}^{k}+ \tau_H^{-1}\df{}{\zeta^{j+1}-\zeta^{j}}.
    \end{align}
    By the subdifferentiability of $G_H^k$ and the descent inequality
    $F_H(\zeta+h) \leq F_H(\zeta) + \pd{\grad F_H(\zeta)}{h} + \frac{L_H}{2}\nr{h}{X}^2$, valid for any $\zeta,h$ (see, e.g., \cite[Theorem 7.1]{clason2020introduction}),
    \[
        \pd{\df{\partial G_{H}^k}{\zeta^{j+1}}+\df{\nabla F_{H}}{\zeta^{j}}}{\zeta^{j+1}-\zeta^{j}}
        \ge
        \df{J_{H}^k}{\zeta^{j+1}}-\df{J_{H}^k}{\zeta^{j}}
        - \frac{L_H}{2}\nr{\zeta^{j+1}-\zeta^{j}}{X_H}^{2}.
    \]
    Applying \(\pd{\,\cdot\,}{\zeta^{j+1}-\zeta^{j}}\) on both sides of \cref{eq:coarse:ei} and using $\epsilon \defeq 2- \tau_HL_H$ thus yields
    \[
        \pd{w_H^k}{\zeta^{j+1}-\zeta^{j}}
        + \frac{\epsilon}{2\tau_H}\nr{\zeta^{j+1}-\zeta^{j}}{X_H}^{2}
        \le
        \df{J_{H}^k}{\zeta^{j}}-\df{J_{H}^k}{\zeta^{j+1}}.
    \]
    Summing  over $j=0,\ldots,m-1$, it follows
    \[
        \pd{w_{H}^{k}}{\zeta^{m}-\zeta^{0}}+\frac{\epsilon}{2\tau_H}\sum _{j=0}^{m-1}\nr{\zeta^{j+1}-\zeta^{j}}{X_H}^{2}
        \le
        \df{J_{H}^{k}}{\zeta^{0}} - \df{J_{H}^{k}}{\zeta^{m}}.
    \]
    Since $\epsilon>0$, by the construction of $w_{H}^{k}$, we get
    using the convexity of $F_H$,
    \[
        \pd{w_{H}^{k}}{\zeta^{m}-\zeta^{0}} \geq \pd{I_{h}^{H} \nabla F(x^k)}{\zeta^{m}-\zeta^{0}}+\df{F_{H}}{\zeta^{0}}-\df{F_{H}}{\zeta^{m}}.
    \]
    Combining these two estimates and simplifying, we obtain the claim.
\end{proof}

\begin{corollary}
    \label{cor:fb:cord}
    Suppose \cref{ass:coarse:basic,ass:coarse:coherence} hold.
    Let $d \defeq I_H^h(\zeta ^{k,m}-\zeta^{k,0})$.
    Then
    \[
        [G+F]'(x^k;d)
        =
        \sup _{g\in \df{\partial G}{x^{k}}}\pd{g+\df{\nabla F}{x^{k}}}{d}
        \leq
        -\frac{\epsilon}{2\mu\tau_H }\sum _{j=0}^{m-1}\nr{\zeta^{k,j+1}-\zeta^{k,j}}{}^{2} \le 0.
    \]
\end{corollary}

\begin{proof}
    By \cref{ass:coarse:coherence},
    $
        \pd{I_h^H g}{\zeta -\zeta ^{k,0}} \leq G_H^k(\zeta)-G_H^k(\zeta^{k,0})
    $
    for all $\zeta$ and $g \in \partial G (x^k)$.
    Taking $\zeta = \zeta^{k,m}$, combining with \cref{lemma:coarse:descent}, we obtain
    \[
        \pd{I_h^H g + I_{h}^{H}\nabla F(x^{k})}{\zeta -\zeta ^{k,0}} +\frac{\epsilon}{2\tau_H}\sum _{j=0}^{m-1}\nr{\zeta^{k,j+1}-\zeta^{k,j}}{X_H}^{2}\leq 0
        \quad \forall\,g \in \partial G (x^k).
    \]
    Since $(I_h^H)^*=\mu I_H^h$, taking the supremum over $g$, we obtain the middle inequality of the claim.
    The equality is standard, e.g., \cite[Lemma 4.4]{clason2020introduction}.
\end{proof}

Now, as $d$ is aligned with a negative subdifferential of the fine-grid objective, we readily show that it is a fine-grid descent direction.

\begin{theorem}
    \label{thm:fb:teosd}
    Suppose \cref{ass:coarse:basic,ass:coarse:coherence} hold.
    Then for $d = I_H^h(\zeta ^{k,m}-\zeta^{k,0})$ and any $\kappa \in (0,1)$, there exists $\theta >0$ such that
    \[
        \label{eq:fb:descent}
        \df{[G+F]}{x^{k}+ \theta d}< \df{[G+F]}{x^{k}}+\kappa\theta \df{{[G+F]}'}{x^{k};d}.
    \]
\end{theorem}

\begin{proof}
    In \cref{cor:fb:cord}, we use the definition of the directional derivative.
\end{proof}

\section{Forward-backward multigrid}
\label{sec:fb}

We now develop the overall multigrid algorithm for \cref{eq:intro:problem:fine-grid}.
Our starting point, again, is the classical forward-backward splitting
\begin{align*}
    x^{k+1} = \proxold{\tau G}{x^k - \tau \nabla F(x^k)}.
\end{align*}
As this is a monotone descent method, \cref{thm:fb:teosd} suggests that performing  coarse-grid iterations between its iterations would not ruin the convergence.
However, details remain to attend to.
To prove convergence, we adapt the technique of \cite{ang2024mgprox}.

\subsection{Algorithm}
\label{sec:full-algorithm}

Our proposed \cref{alg:fb:mg} is, for simplicity, limited to two grid levels, but easily extended to multiple levels.
It depends on \emph{line search} to guarantee \eqref{eq:fb:descent}, as well as a \emph{trigger condition} to perform coarse corrections. Options for the latter include:
\begin{enumerate}[nosep]
    \item heuristic, after enough progress since the previous correction \cite{parpas2017multilevel};
    \item a fixed number of fine iterations between coarse corrections; or
    \item a bounded number of coarse corrections during the algorithm runtime.
\end{enumerate}
Although a standard line search procedure can be used, in practise, for efficiency, we either take at each coarse correction a fixed $\theta_{k} =\bar\theta_k$, or, if this does not satisfy the sufficient decrease condition  \eqref{eq:fb:descent}, reject the coarse grid result with $\theta_k=0$.

\subsection{Convergence}
\label{sec:convergence}

\begin{algorithm}[t]
    \caption{Forward-backward multigrid (FBMG)}
    \label{alg:fb:mg}
    \begin{algorithmic}[1]
        \Require $F,G,F_H$ and $\tau,\tau_H>0$ satisfying \cref{ass:fb:general,ass:coarse:basic}. Sufficient descent parameter $\kappa \in (0, 1)$ as well as a line search procedure and trigger condition.
        \State Choose an initial iterate $x^0 \in X$. Set $J \defeq G+F$.
        \ForAll{$k=0,1,2,\ldots$ until a chosen stopping criterion is fulfilled}
            \If{a trigger condition is satisfied}
                \State Choose an initial coarse point $\zeta^{k,0}$ (e.g., $I_h^H x^k$).
                \State Design $G_H^k$ satisfying the nonsmooth coherence condition (\cref{ass:coarse:coherence})
                \State Set
                $
                    w_{H}^{k} \defeq I_h^H\nabla F(x^k) - \df{\nabla F_{H}}{\zeta ^{k,0}},
                $
                \ForAll{$j=1,\ldots,m-1$}
                    \State $\zeta^{k,j+1} \defeq \proxold{\tau _H G_H^k}{\zeta^{k,j}-\tau _H [\nabla F_H(\zeta^{k,j}) + w_H^k]}$
                \Comment{Coarse-grid FB}
                \EndFor
                \State Set $d \defeq I_H^h(\zeta ^{k,m}-\zeta ^{k,0})$.
                \State Find $\theta_k \ge 0$ such that {\(\df{J}{x^{k}+\theta_k d} \le \df{J}{x^{k}}+\kappa \theta_k\df{{J}'}{x^{k};d}\)}.
                \label{line:fb:mg:linesearch}
                \Comment{Line search}
                \State
                $x^{k+1} \defeq \proxold{\tau G}{z^k + \nabla F(z^k)}$
                \quad\text{for}\quad $z^k \defeq x^k + \theta_k d$
                \label{line:fb:mg:fine-fb1}
                \Comment{Fine-grid FB}
            \Else
                \State\label{line:fb:mg:fine-fb2}
                \smash{$x^{k+1} \defeq \proxold{\tau G}{x^k + \nabla F(x^k)}$}
                \Comment{Fine-grid FB}
            \EndIf
        \EndFor
    \end{algorithmic}
\end{algorithm}

\begin{assumption}
    \label{ass:fb:general}
    $F,G: X \to \extR$ are proper, convex and lower semicontinuous, $F$ Fréchet differentiable with $L$-Lipschitz gradient.
    The step length $\tau \in (0, 1/L)$.
\end{assumption}

\begin{theorem}[Sublinear convergence]
    \label{thm:fb:teoconvsub}
    Suppose \cref{ass:fb:general} holds, and $x^{*}\in \df{\left[\partial J\right]^{-1}}{0}$ where $J:= G+F$.
    Let $\{x^k\}_{k \ge 1}$ be generated by \cref{alg:fb:mg} for an initial $x^0\in X$ such that the corresponding sublevel set is bounded, i.e.,
    \begin{gather}
        \nonumber
        \varepsilon \defeq \diam(\sublev_{J(x^0)}J) := \sup \{ \nr{x-y}{2} \mid J(x), J(y) \le J(x^0)\} < \infty.
    \shortintertext{Then}
        \label{eq:fb:prop}
        J(x^{k+1})-J(x^*) \leq \frac{1}{k}\max \cj{4C,J(x^0)-J(x^*)}
        \quad\text{for}\quad
        C = \frac{2\varepsilon ^2}{\tau(2-\tau L)}>0.
    \end{gather}
\end{theorem}

\begin{proof}
    \Cref{thm:fb:teosd} guarantees the line search on \cref{line:fb:mg:linesearch} of \cref{alg:fb:mg} to be satisfiable (strictly for a $\theta_k>0$, although we also allow $\theta_k=0$ and mere non-increase).
    By standard arguments based on convexity, the descent lemma, and the Pythagoras identity (see the proof of \cite[Theorem 11.4]{clason2020introduction}), \cref{line:fb:mg:fine-fb1} satisfies
    \begin{align}
        \label{eq:fb:descent3}
        J(x^{k+1}) -J(x^*) +\frac{1}{2\tau} \nr{x^{k+1}-x^*}{X}^2 + \frac{1-\tau L}{2\tau}\nr{x^{k+1}-z^k}{X}^2\leq \frac{1}{2\tau} \nr{z^k-x^*}{X}^2.
    \end{align}
    Since $\tau L<1$ we have
    \[
        \begin{split}
            J(x^{k+1}) -J(x^*)
            &
            \leq \frac{1}{2\tau} \df{}{\nr{z^k-x^*}{X}^2-\nr{x^{k+1}-x^*}{X}^2}
            \\
            &
            = \frac{1}{2\tau} \df{}{\nr{z^k-x^*}{X}-\nr{x^{k+1}-x^*}{X}}\df{}{\nr{z^k-x^*}{X}+\nr{x^{k+1}-x^*}{X}}
            \\
            &
            \leq \frac{\varepsilon}{\tau} \df{}{\nr{z^k-x^*}{X}-\nr{x^{k+1}-x^*}{X}}
            \\
            &
            \leq \frac{\varepsilon}{\tau} \nr{x^{k+1}-z^k}{X}.
        \end{split}
    \]
    Rearranging gives
    \begin{equation}
        \label{eq:fb:ecu1}
        \df{}{J(x^{k+1})-J(x^*)}^2 \leq (\varepsilon \tau ^{-1})^2 \nr{x^{k+1}-z^k}{X}^2.
    \end{equation}
    On the other hand, taking $x^*=z^k$ in \eqref{eq:fb:descent3}, and continuing with $J(z^k) \le J(x^k)$ established by the line search procedure on \cref{line:fb:mg:linesearch} of \cref{alg:fb:mg}, we obtain
    \[
        \frac{2-\tau L}{2\tau}\nr{x^{k+1}-z^k}{X}^2
        \le
        J(z^k) - J(x^{k+1})
        \le
        J(x^k) - J(x^{k+1}).
    \]
    Combining with \cref{eq:fb:ecu1} yields
    $
        (J(x^{k+1})-J(x^*))^2\leq C[J(x^k) - J(x^{k+1})].
    $
    Repeating the analysis with $z^k=x^k$ establishes the same result for \cref{line:fb:mg:fine-fb2}.
    Now, according to \cite[Lemma 4]{karimi2017imro}, the monotonically decreasing sequence $\cj{J(x^k)}_{k\in \mathbb{N}}$ satisfies \cref{eq:fb:prop}.
\end{proof}

\section{Total variation regularised imaging problems}
\label{sec:tv}

We will in \cref{sec:numerical} apply \cref{alg:fb:mg} to image processing problems of the form
\begin{equation}
    \label{eq:tv:primal-problem}
    \min _{y\in \R^n}~ \phi(y) + \alpha \nr{\nabla_h y}{2,1}
    \quad\text{for}\quad
    \df{\phi}{y} := \frac{1}{2}\sum _{s=1}^t \nr{T_sy-b_s}{2}^2,
\end{equation}
where $\grad_h \in \linear(\R^n; \R^{D \times n})$ for some dimension $D$.
Since the nonsmooth total variation regulariser is not prox-simple, to derive an efficient method, we will need to work with the dual problem.
Since MRI involves complex numbers, we allow $T_s \in \linear(\C^n; \C^n)$ and $b_s \in \C^n$.
Thus the dual formulation is
\begin{equation}
    \label{eq:tv:dual-problem}
    \min _{x\in \R^{D\times n}}\, \phi ^*(-\nabla_h ^* x) + G(x)
    \quad\text{for}\quad
    \df{G}{x} \defeq (\alpha\nr{\cdot}{2,1})^*(x) = \sum_{i=1}^n \delta _{B(0,\alpha)}(x_{\freevar i}).
\end{equation}
We construct $\phi^*$ in \cref{sec:tv:data-fenchel}, after we have first constructed the coarse nonsmooth function $G_H^k$ and its proximal operator in \cref{sec:tv:coarse,sec:tv:coarse-prox}.

\subsection{The coarse problem}
\label{sec:tv:coarse}

We use the standard restriction operator\footnote{For MRI we could replace $I_H^h$ by $\mathcal{F}^*I_H^h\mathcal{F}$, where $\mathcal{F}$ is the Fourier transform in the relevant grid, but do not currently use this form.} $I_h^H = R  \otimes R$, where, in stencil notation,
$R = \begin{bsmallmatrix}
        \frac{1}{2}& 1 & \frac{1}{2}
       \end{bsmallmatrix};$
see \cite{briggs2000multigrid}.
The prolongation operator is then $I_h^H = \frac{1}{4} (I_H^h)^*$.
We take
\begin{equation}
    \label{eq:tv:gHk}
    G_H^k(\zeta) \defeq \sum_{l=1}^N \delta_{\Omega_l}(\zeta_{\freevar l})
    \quad\text{for}\quad
    \Omega_l \defeq \zeta_{\freevar l}^{k,0} + \polar \Gamma_l,
\end{equation}
where for all coarse pixel indices $l=1,\ldots,N$, we define
\begin{equation}
    \label{eq:tv:gammal}
    \Gamma_l
    \defeq
    [I_h^H \partial G(x^k)]_l
    =
    \left\{
        \sum\nolimits_{p\in A_l}q^k_{\freevar p}
        \,\middle|\,
        q_{\freevar p}^k \in [\partial G(x^k)]_p\ \forall p\in A_l
    \right\},
\end{equation}
with $A_l \subset \{1,\ldots,n\}$ the subset of fine pixel indices $i$ that contribute to the coarse pixel $l$ via $I_h^H$, that is, $[I_h^H]_{li} \ne 0$.
Note that $\Omega_l$ is nonempty, closed and convex for all $l=1,\ldots,N$.
Apart from the Lipschitz gradient, there are no theoretical restrictions on $F_H$, but we
take it as a coarse version of $\phi ^*\circ -\nabla _h^*$, as we describe later in \cref{ex:coarse:FH}, after forming $\phi^*$.

\begin{lemma}
    \Cref{ass:coarse:coherence} holds for $G$ and $G_H^k$ as in \cref{eq:tv:gHk,eq:tv:dual-problem}.
\end{lemma}

\begin{proof}
    We recall for all fine pixel indices $i=1,..,n$ that
    \begin{equation}
        \label{eq:cap4:subdGpsr}
        \left[\partial \df{G}{x^k}\right]_i
        =\df{\partial \delta _{\df{B}{0,\alpha}}}{x^k_{\freevar i}}=\left\{\begin{array}{ll}
        \cj{0}, & x^k_{\freevar i} \in \interior B(0, \alpha),  \\
        \cj{\beta x^k_{\freevar i} \mid \beta\geq 0}, & x^k_{\freevar i} \in \boundary B(0, \alpha), \\
        \emptyset & \text{otherwise}.
        \end{array}\right.
    \end{equation}
    It follows that $\Gamma_l$ is a (possibly non-convex) cone in $\R^D$.
    We then deduce
    \[
        \partial \delta_{\Omega_l}(\zeta_{\freevar l}^{k,0})
        =
        N_{\Omega_l}(\zeta_{\freevar l}^{k,0})
        =
        N_{\Gamma_l^\circ}(0)
        = (\Gamma_l^\circ)^\circ \supset \Gamma_l,
    \]
    and further
    \[
        I_h^H \partial G(x^k)
        = \Gamma_1 \times \cdots \times \Gamma_N
        \subset
        \partial \delta _{\Omega _l}(\zeta _{\freevar 1}^{k,0})
        \times \cdots \times
        \partial \delta _{\Omega _l}(\zeta _{\freevar N}^{k,0})
        = \partial G_H^k(\zeta ^{k,0}).
        \qedhere
    \]
\end{proof}

\subsection{The coarse proximal operator}
\label{sec:tv:coarse-prox}

By \eqref{eq:tv:gHk}, $\Omega_l - \zeta_l^{k,0} = \polar\Gamma_l = \polar{(\bipolar\Gamma_l)}$ is a closed convex cone for each $l$.
When $D=2$, it is therefore defined by at most two “director” vectors.
If we can identify them, it will be possible to write the proximal operator of $G_H^k$ in a simple form.

\begin{lemma}
    \label{lemma:tv:gamma}
    Suppose $D=2$.
    For any coarse pixel index $l \in \{1,\ldots,N\}$, $\bipolar \Gamma_l$ can only take for some fine pixel indices $j,s \in A_l$ one of the forms
    \[
        \cj{0},\quad \ccone\cj{x_{\freevar j}},\quad \ccone\cj{x_{\freevar j},x_{\freevar s}},\quad \ccone\cj{x_{\freevar j},x_{\freevar s},-x_{\freevar j}},\quad \R^2.
    \]
\end{lemma}

\begin{proof}
    We construct the director vectors algorithmically, starting with empty sets $V_l$ and $O_l$. The former will eventually contain the directors, while $O_l$ tracks subspaces generated by opposing vectors.
    For all $p \in A_l$, we repeat steps 1 and 2:
    \begin{enumerate}
        \item
        We omit $x^k_{\freevar p} \in  \interior \df{B}{0,\alpha}$ since, by \cref{eq:tv:gammal}, the subgradient $q_{\freevar p}^k=0$, does not contribute to $\Gamma_l$.
        By contrast, if $x^k_{\freevar p} \in \boundary \df{B}{0,\alpha}$, we add this vector to $V_l$.

        \item
        If $V_l=\{z_{j}, z_{s}, z_{p}\}$ for three distinct vectors, one of them must be superfluous for forming $\polar\Gamma_l$.
        To discard it, we form the linear system
        $
            z_{j} \beta_1 + z_{s} \beta_2 = z_{p}
        $
        for the unknowns $\beta_1$ and $\beta_2$, and consider several cases:
        \begin{enumerate}[label=(\alph*),nosep]
            \item  For a non-unique solution, $z_{s} = c z_{j}$ for some $c \ne 0$.
            When $c > 0$, $z_s$ is superfluous.
            When $c < 0$, $\Gamma_l$ must be contained in the subspace orthogonal to $z_s$.
            In both cases, we take $V_l = \{z_{j},z_{p}\}$, and in the latter add $z_s$ to $O_l$.
        \end{enumerate}
        Otherwise the system has a unique solution, and we continue with the cases:
        \begin{enumerate}[resume*]
            \item If $\beta _1 \geq 0$ and $\beta _2 \geq 0$, then $z_{p} \in \ccone\{z_{j}, z_{s}\}$, so we remove $z_{p}$ from $V_l$.
            \item If $\beta _1 >0$ and $\beta_2 <0$, then $z_{j} \in \ccone \{z_{s}, z_{p}\}$, so we remove $z_{j}$ from $V_l$.
            \item If $\beta _1 <0$ and $\beta_2 >0$, then $z_{s} \in \ccone\{z_{j}, z_{p}\}$, so we remove $z_{s}$ from $V_l$.
            \item If $\beta _1 = 0$ and $\beta_2 <0$ or $\beta _1 < 0$ and $\beta_2 = 0$, then $z_{s} = -z_{p}$ or $z_{j} = -z_{p}$, so we eliminate $z_{p}$ from $V_l$ but add it to $O_l$.
            \item If $\beta _1 < 0$ and $\beta_2 <0$, then $\Gamma_l^\circ=\{0\}$, so we terminate with $\bipolar \Gamma_l = \rea{2}$.
        \end{enumerate}
    \end{enumerate}

    If the construction did not terminate in the above loop, we consider:
    \begin{enumerate}[label=(\roman*)]
        \item If $\abs{V_l}+\abs{O_l} \le 2$, then by construction $O_l=\emptyset$.
        If $V_l=\emptyset$, also $\bipolar \Gamma_l = \cj{0}$.
        Otherwise $\bipolar \Gamma_l = \ccone V_l$.
        \item If $|V_l| +|O_l| = 3$, then $O_l=\{z_p\}$ and $V_l=\{z_{j}, z_{s}\}$ with $z_{p} = -z_{j}$ or $z_{p} = -z_{s}$. In other words, the three vectors define the half-space $\bipolar \Gamma_l = \ccone\cj{z_{j},z_{s},-z_{j}}$.
        \item $|V_l| +|O_l| \geq 4$ then $O_l$ defines two distinct subspaces that contain $\polar \Gamma_l$, which must then by $\{0\}$. Hence $\bipolar \Gamma _l = \rea{2}$.
        \qedhere
    \end{enumerate}
\end{proof}

In the next proposition, $z_{j}^\ortho$ denotes a vector orthogonal to $z_{j}$ with $\sup_{z \in \Gamma} \iprod{z}{z_j^\ortho} \le 0$.
We also write
$
    p(\zeta, z) \defeq \max \cj{0,\pd{\zeta}{z}}z/\nr{z}{2}^2
$
for the projection of $\zeta$ to $z[0,\infty)$.

\begin{proposition}
    \label{prop:tv:prox}
    The proximal operator of $G_H^k$ defined in \eqref{eq:tv:gHk} is given by
    \begin{subequations}%
    \begin{equation}
        \label{eq:tv:prox1}
        [\proxold{\gamma G_H^k}{\zeta}]_l = \left \{\begin{array}{ll}
            \zeta _l, & \bipolar \Gamma _l = \cj{0}, \\
            \zeta _l -  p(\zeta_l - \zeta _l^{k,0},z_{j}), & \bipolar \Gamma _l = \ccone\cj{z_{j}}, \\
            \zeta _l^{k,0} + p(\zeta_l - \zeta _l^{k,0},z_j^\ortho), & \bipolar \Gamma _l = \ccone\cj{z_{j},z_{s},-z_{j}}, \\
            \zeta _l^{k,0}, & \bipolar \Gamma_l = \rea{2}, \\
            \text{see below} &\bipolar  \Gamma_l = \ccone\cj{z_{j},z_{s}},
        \end{array}\right.
    \end{equation}
    for each component $l=1,\ldots,m$, where in the final case,
    \begin{equation}
        \label{eq:tv:prox2}
        [\proxold{\gamma G_H^k}{\zeta}]_l =  \left \{\begin{array}{ll}
            \zeta _l^{k,0}, & \zeta _l\in \ccone\cj{z_{j},z_{s}}, \\
            \zeta _l, & \zeta _l\in \ccone\cj{z_{j}^\ortho,z_{s}^\ortho}, \\
            \zeta _l^{k,0} + p(\zeta_l - \zeta _l^{k,0},z_{j}^\ortho), & \zeta _l\in \ccone\cj{z_{j},z_{j}^\ortho}, \\
            \zeta _l^{k,0} + p(\zeta_l - \zeta _l^{k,0},z_{s}^\ortho), & \zeta _l\in \ccone\cj{z_{s},z_{s}^\ortho}.
        \end{array}\right.
    \end{equation}%
    \end{subequations}%
\end{proposition}

\begin{figure}[t]
    \centering
    \includegraphics[width=0.4\textwidth]{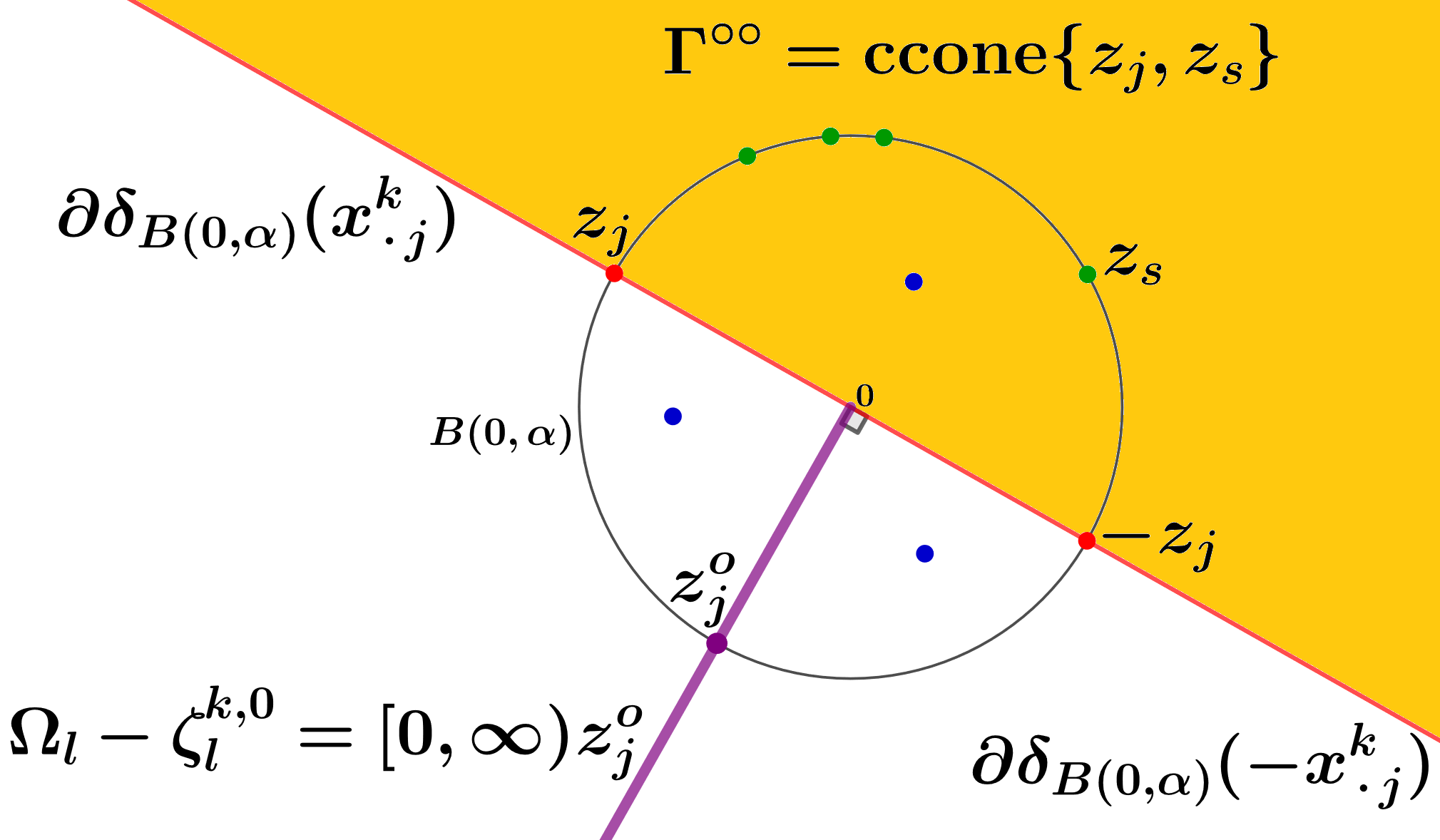}
    \hfil
    \includegraphics[width=0.37\textwidth]{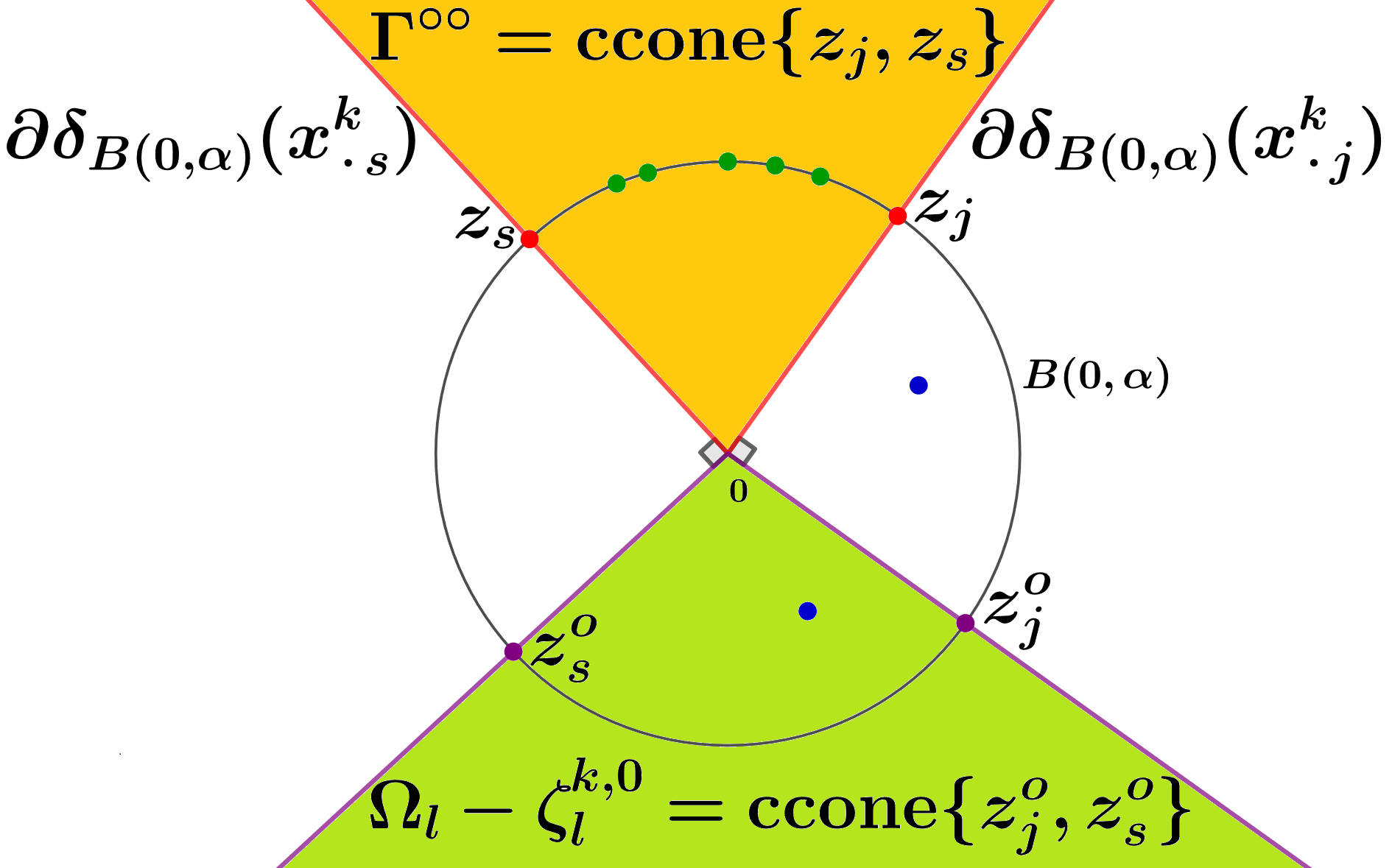}
    \caption{Illustration of the middle case of \eqref{eq:tv:prox1} (left) and \eqref{eq:tv:prox2} (right).}
    \label{fig:prox}
\end{figure}

\begin{proof}
    The proximal map of $G_H^k$ separates into individual Euclidean projections onto each $\Omega_l$.
    These are determined by $\bipolar \Gamma_l$, so we consider the cases of \cref{lemma:tv:gamma}:
    \begin{enumerate}[label=(\alph*)]
        \item
        If $\bipolar \Gamma_l = \cj{0}$, we have $\Omega_l = \rea{2}$.
        Hence the proximal map is the identity.

        \item
        If $\bipolar \Gamma _l = \rea{2}$, we have $\Omega _l = \cj{\zeta _l^{k,0}}$, so the proximal map is the constant $\zeta _l^{k,0}$.

        \item
        If $\bipolar \Gamma_l = \ccone\cj{z_{j}}$, then
        $
            \Omega _l
             = \zeta_l^{k,0} + \polar{(\bipolar \Gamma_l)}
             = \zeta_l^{k,0} + \polar{\{z_j\}}.
        $
        Thus $[\proxold{\gamma G_H^k}{\zeta}]_l = \zeta _ l -p(\zeta _l - \zeta _l^{k,0},z_{j})$.

        \item
        If $\bipolar \Gamma _l = \ccone\cj{z_{j},z_{s},-z_{j}}$, then, likewise, $\Omega _l =  \zeta _l^{k,0} + [0, \infty) z_j^\ortho$, resulting in
        $[\proxold{\gamma G_H^k}{\zeta}]_l = \zeta _ l^{k,0} -p(\zeta _l - \zeta _l^{k,0},z_{j}^\ortho)$; see \cref{fig:prox}.

        \item
        When $\bipolar \Gamma_l = \ccone \cj{z_{j},z_{s}}$, we can divide $\R^2$ into four cones, each with distinct projection operator agreeing with \cref{eq:tv:prox2}; see \cref{fig:prox}.
        \qedhere
    \end{enumerate}
\end{proof}

\subsection{The Fenchel conjugate of the (complex) data term}
\label{sec:tv:data-fenchel}

To form the Fenchel conjugate of $\phi$, we start with its derivative.
We write $\norm{\freevar}_{\R^n}$ for the Euclidean norm in $\R^n$ and
$\norm{x}_{\C^n} \defeq \sqrt{\sum_{k=1}^n \abs{x_i}^2}$ for $x=(x_1,\ldots,x_n) \in \C^n$.

\begin{lemma}
    \label{lemma:tv:phi-derivative}
    $\nabla \phi (y) = Ty -e$ for $T \defeq \sum _{s=1}^t \Re T_s^*T_s$ and $e \defeq \sum _{s=1}^t \Re T_s^* b_s$.
\end{lemma}

\begin{proof}
    Due to properties of the complex inner product, for any $y, h \in \R ^n$,
    \begin{align*}
        \nr{T_s(y+h) -b_s}{\C^n}^2 - \nr{T_s y -b_s}{\C^n}^2
        &
        = 2 \Re\pd{T_sy-b_s}{T_sh}_{\mathbb{C}^n} + \nr{T_s h}{\C^n}^{2}
        \\
        &
        =
        2 \iprod{\Re T_s^*(T_sy-b_s)}{h}_{\R^n} + \nr{T_s h}{\C^n}^{2}.
    \end{align*}
    Dividing by $2$ and summing over $s=1,\ldots,t$, therefore
    \[
        \df{\phi}{y+h}-\df{\phi}{y}
        =
        \iprod{Ty-e}{h}_{\R^n} + \frac{1}{2} \sum _{s=1}^t \nr{T_s h}{\C^n}^{2}.
    \]
    By the definition of the Fréchet derivative, the claim follows.
\end{proof}

From the definition of the Fenchel conjugate and the Fermat principle, now
\begin{equation}
    \label{eq:tv:phi-conjugate-0}
    \phi ^*(z) = \pd{z}{T^{-1}(z+e)} - \phi(T^{-1}(z+e)).
\end{equation}
We next develop a more tractable form, which shows \eqref{eq:tv:dual-problem} to have equivalent form
\begin{equation}
    \label{eq:tv:equiv-dual}
    \min _{x\in \R^{D\times n}} \frac{1}{2} \nr{T^{-1/2}(\nabla _h^*x-e)}{2}^2 + \sum _i \delta_{B(0,\alpha)}(x_{\freevar i}).
\end{equation}

\begin{lemma}
    \label{lemma:tv:complex-operator-norm}
    If $A \in \linear(\C^n; \C^n)$.
    Then $\nr{Ay}{\C^n}^2 = \pd{\Re(A^*A)y}{y}_{\R^n}$ for all $y\in \R^n$.
\end{lemma}
\begin{proof}
    Writing $A = A_1 +iA_2$ for $A_1,A_2 \in \rea{n\times n}$, we have $A^* = A_1^T -iA_2^T$.
    Thus
    \begin{align*}
        \nr{Ay}{\C^n}^2 &
        =
        \pd{A^*Ay}{y}_{\mathbb{C}^n}
        =\pd{(A_1^T -iA_2^T)(A_1 +iA_2)y}{y}_{\mathbb{C}^n}
        \\
        & = \pd{(A_1^TA_1 + A_2^TA_2)y}{y}_{\R^n} + i\pd{(A_1^TA_2 - A_2^TA_1)y}{y}_{\R^n}.
    \end{align*}
    The last term is zero by the properties of the real inner product and transpose.
\end{proof}

\begin{lemma}
    Let $r_s = b_s - T_sT^{-1}e$ for $T$ and $e$ from \cref{lemma:tv:phi-derivative}.
    Then
    \[
        \phi^*(z)
        = \frac{1}{2} \nr{T^{-1/2}(z+e)}{\R^n}^2 - \frac{1}{2}\sum _{s=1}^t\nr{r_s}{\C^n}^2 - \nr{T^{-1/2}e}{\R^n}^2
        \quad\forall z \in \R^n.
    \]
\end{lemma}

\begin{proof}
    \Cref{lemma:tv:complex-operator-norm} and the properties of the complex inner product yield
    \begin{equation}
        \label{eq:tv:conjugate:0}
        \begin{aligned}[t]
        \nr{T_s(T^{-1}(z+e))-b_s}{\C^n}^2
        &
        = \nr{T_sT^{-1}z - r_s}{\C^n}^2
        \\
        &
        = \nr{T_sT^{-1}z}{\C^n}^2 - \pd{T_s T^{-1}z}{r_s}_{\mathbb{C}^n} - \pd{r_s}{T_s T^{-1}z}_{\mathbb{C}^n} +\nr{r_s}{\C^n}^2
        \\
        &
        = \iprod{\Re(T_s^*T_s) T^{-1}z}{T^{-1}z}_{\R^n}
        - 2 \Re \pd{T^{-1}z}{T_s^* r_s}_{\mathbb{C}^n}
        + \nr{r_s}{\C^n}^2.
        \end{aligned}
    \end{equation}
    We have $\sum_{s=1}^t T_s^*r_s=\sum_{s=1}^t T_s^* b_s - \sum_{s=1}^t T_s^*T_s T^{-1} e$, hence $\Re \sum_{s=1}^t T_s^*r_s= 0$.
    Since $T^{-1}z \in \R^n$, it follows that $\sum _{s=1}^t\Re \pd{T^{-1}z}{T_s^* r_s}_{\mathbb{C}^n} = 0$.
    Dividing \eqref{eq:tv:conjugate:0} by $2$ and summing over $s=1,\ldots,t$, therefore
    \[
        \phi(T^{-1}(z+e))
        = \frac{1}{2}\pd{z}{T^{-1}z}_{\R^n} + \frac{1}{2}\sum _{s=1}^t\nr{r_s}{\C^n}^2.
    \]
    Using this expression in \eqref{eq:tv:phi-conjugate-0}, the claim readily follows.
\end{proof}

We can now finally suggest one way to form the coarse function $F_H$:

\begin{example}
    \label{ex:coarse:FH}
    Form a coarse discrete gradient operator $\grad_H \in \linear(\R^N; X_H)$ and set
    \[
        F_H(\zeta ) \defeq \frac{1}{2} \nr{T_H^{-1/2}(\grad _H ^*\zeta - b_H)}{\R^N}^2
        \quad\text{for}\quad
        T_H \defeq I_h^H T
        \quad\text{and}\quad
        b_H \defeq I_h^H b.
    \]
\end{example}

\section{Numerical experience}
\label{sec:numerical}

We now report our numerical experience with denoising and MRI.
Both problems have the primal form \eqref{eq:tv:primal-problem}.
We work with the equivalent dual problem \eqref{eq:tv:dual-problem}.
Our Julia implementation is available on Zenodo \cite{multigrid-codes-zenodo}.

\subsection{Denoising}
\label{sec:numerical:denoising}

\makeatletter
\def\hlinewd#1{%
\noalign{\ifnum0=`}\fi\hrule \@height #1
\futurelet\reserved@a\@xhline}
\makeatother

\begin{table}[t]
    \caption{Time (seconds) to reach relative error $\relerr_1$ and $\relerr_2$ for both experiments.}
    \label{tab:error-time}
    \centering
    \begin{tabular}{l@{\quad}|@{\quad}lll@{\quad}|@{\quad}lll}
        \hlinewd{1pt}
        Experiment & $\relerr _1$ & FB & FBMG & $\relerr _2$ & FB & FBMG \\
        \hline
        Denoising & 0.01 & 0.030498 & 0.007404 & 0.001 & 0.231761 & 0.102651 \\
        MRI & 0.01  & 0.069281 & 0.005343 & 0.001  & 0.142469  & 0.046784 \\
        \hlinewd{1pt}
    \end{tabular}
\end{table}

For denoising we use one full sample, i.e., solve \eqref{eq:tv:primal-problem} with $t=1$ and $T_s=I$.
The dual problem \eqref{eq:tv:equiv-dual} is then $\min_{x\in \R^{D\times n}} \nr{\nabla _h^* x - e}{\R^n}^2 +\sum _{i} \delta _{B(0,\alpha)}(x_{\freevar ,i})$.
We use the the \emph{Blue Marble} public domain test image with resolution $3002\times 3000$. We add pixelwise Gaussian noise of standard deviation $\sigma = 0.4$.
The Lipschitz constant $L=L_H=8$ \cite{chambolle2004algorithm}.
We take $\alpha = 0.85$ and $\tau = 0.95/L$ and $\tau_H = 1.95/L$.
In FBMG, we perform $m=6$ coarse steps, based on trial and error, before the first 110 fine iterations only.
For line search, we try $\theta_k = \bar\theta_k \defeq \omega_k\pd{T^{-1}(e-\grad ^* x^{k})}{\grad^* d}/ \nr{T^{-1/2}\grad ^*d}{2}^2 \ge 0$ for the scaling factor $\omega_k = 2/5$, and otherwise fail with $\theta_k=0$.\footnote{When $\omega_k <2$, $\bar\theta_k$ is a scaled-down exact solution to \eqref{eq:fb:descent} for $F=\phi^*$ and $G=0$. Small $\omega_k$ attempts to ensure $x^k+\bar\theta_k d \in B(0, \alpha)^n$. By convexity, this check guarantees descent.}
We illustrate the data and reconstructions in \cref{fig:denoising:compare:FB:FBMG}, and the performance in \cref{fig:denoising:graphs,tab:error-time}, where,
for $x^*$ computed by 100000 iterations of FBMGs, the \emph{relative error}
\begin{equation}
    \label{eq:relerror}
    \relerr=\relerr^k \defeq (v(x^k)-v(x^*))/(v(x^0)-v(x^*))
    \quad\text{with}\quad
    v(x) \defeq \phi^*(-\nabla _h^* x) + G(x).
\end{equation}
The \emph{iteration comparison number} in \cref{fig:denoising:graphs} scales coarse iterations by the ratio of the number of coarse to fine pixels.

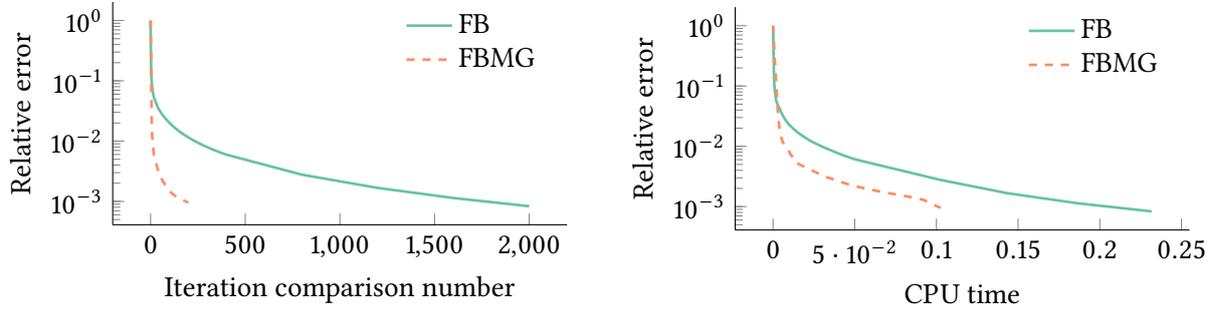
\begin{figure}[t]
    \centering
    \begin{subfigure}{0.48\columnwidth}
    \begin{tikzpicture}
        \begin{axis}[%
            width = \linewidth,
            height = 0.6\linewidth,
            axis x line*=bottom,
            axis y line*=left,
            xlabel={Iteration comparison number},
            ylabel={Relative error},
            ymode=log,
            xmode=normal,
            legend pos = north east,
            ]

            \addplot [fb] table[x=iter,y=relative]{fb_simple_final_rel.txt};
            \addlegendentry{FB}

            \addplot [projgrad] table[x=iter,y=relative]{fbmg_simple_final_rel.txt};
            \addlegendentry{FBMG}
        \end{axis}
    \end{tikzpicture}
    \end{subfigure}
    \hfill
    \begin{subfigure}{0.48\columnwidth}
    \begin{tikzpicture}
        \begin{axis}[%
            width = \linewidth,
            height = 0.6\linewidth,
            axis x line*=bottom,
            axis y line*=left,
            xlabel={CPU time},
            ylabel={Relative error},
            ymode=log,
            xmode=normal,
            legend pos = north east,
            ]

            \addplot [fb] table[x=cputime,y=relative]{fb_simple_final_rel.txt};
            \addlegendentry{FB}

            \addplot [projgrad] table[x=cputime,y=relative]{fbmg_simple_final_rel.txt};
            \addlegendentry{FBMG}
        \end{axis}
    \end{tikzpicture}
    \end{subfigure}
    \caption{Relative error \eqref{eq:relerror} versus iteration count and CPU time for denoising.}
    \label{fig:denoising:graphs}
\end{figure}

\begin{figure}[t]
    \centering
    \begin{subfigure}{.2\textwidth}
        \centering
        \includegraphics[width=0.95\linewidth]{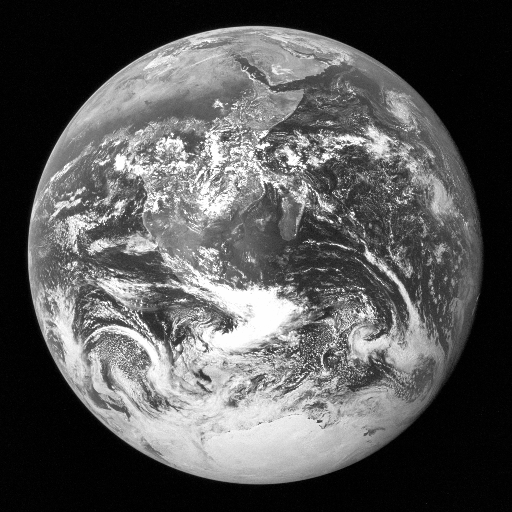}
        \caption{Original}
        \label{fig:original:image}
    \end{subfigure}%
    \begin{subfigure}{.2\textwidth}
        \centering
        \includegraphics[width=0.95\linewidth]{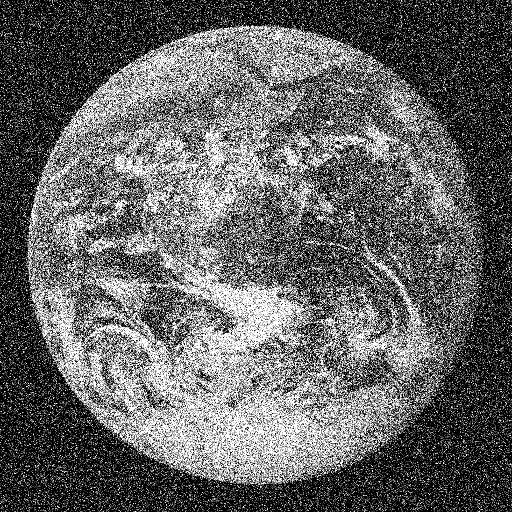}
        \caption{Noisy}
        \label{fig:noisy:image}
    \end{subfigure}
    \begin{subfigure}{.2\textwidth}
        \centering
        \includegraphics[width=0.95\linewidth]{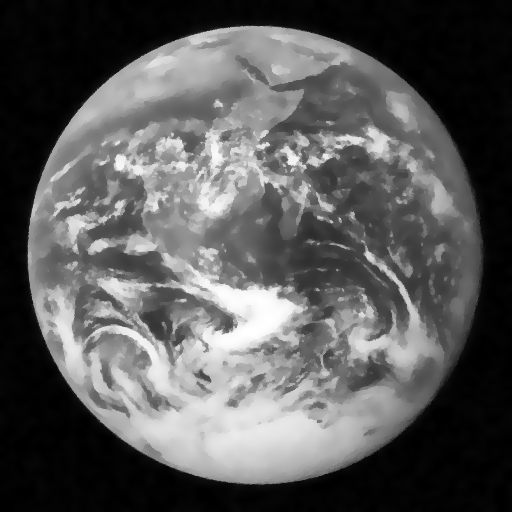}
        \caption{FB}
        \label{fig:fb:solution:image}
    \end{subfigure}%
    \begin{subfigure}{.2\textwidth}
        \centering
        \includegraphics[width=0.95\linewidth]{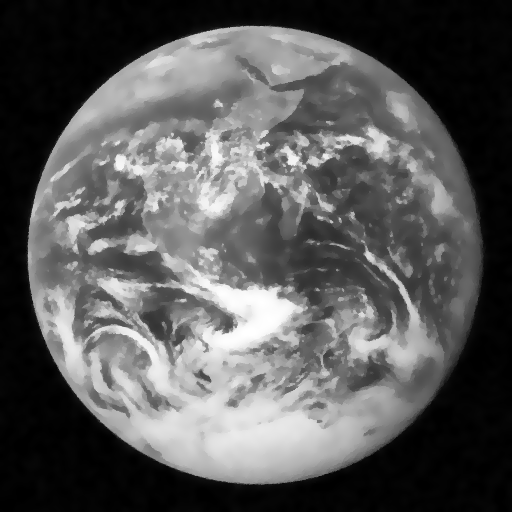}
        \caption{FBMG}
        \label{fig:fbmg:solution:image}
    \end{subfigure}%
    \caption{Denoising data and results at relative error $\relerr= 0.001$.
    }
    \label{fig:denoising:compare:FB:FBMG}
\end{figure}

\subsection{Magnetic resonance imaging}
\label{sec:numerical:mri}

We take $T_s \defeq S_s\mathcal{F}$, where $\mathcal{F}$ is the discrete Fourier transform and $S_s$ is a frequency subsampling operator. Then $T$ of \cref{lemma:tv:phi-derivative} becomes
$
    T = \Re \mathcal{F}^* S \mathcal{F} = \mathcal{F}^* \Sym S \mathcal{F}
$
for $S \defeq \sum _{s=1}^t S_s^* S_s$ and $\Sym S$ its symmetrisation over positive and negative frequencies in both axes.
Thus $T^{-1}$, required for the dual problem \eqref{eq:tv:equiv-dual}, exists and is easily calculated when each frequency is sampled by some $S_s$.
With $t=21$, we form each subsampling mask $S_1,\ldots,S_t$ by random sampling 150 lines in the Fourier space from a uniform distribution of such subsets of lines.
We use the MRI phantom of \cite{belzunce2018high} with resolution $583\times 493$, and add complex Gaussian noise with standard deviation $\sigma = 50$.
We take $\alpha = 1.15$, and the step length parameters and line search as for denoising with $L = 8\norm{T^{-1}}$, $L_H = 8 \norm{T_H^{-1}}$.
We perform $m=6$ coarse steps before the first 500 fine iterations only.
We illustrate the data, reconstructions, and performance in \cref{fig:mri:compare:FB:FBMG:mri,fig:mri:graphs,tab:error-time}.

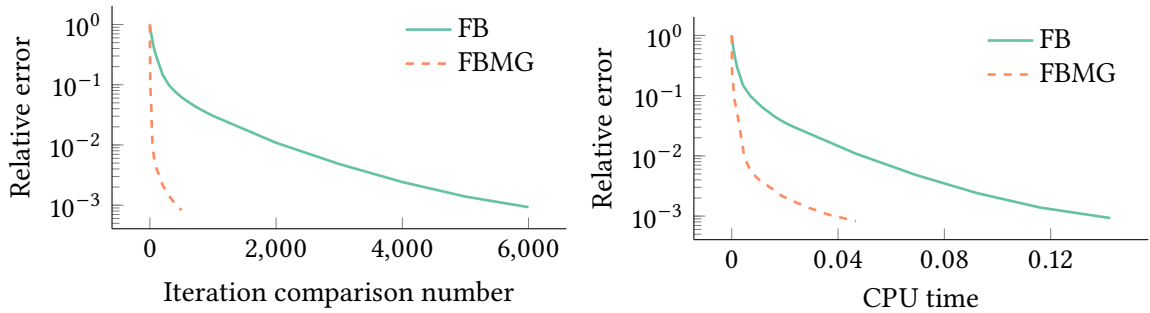
\begin{figure}[t]
    \centering
    \begin{subfigure}{0.48\columnwidth}
    \begin{tikzpicture}
        \begin{axis}[%
            width = \linewidth,
            height = 0.6\linewidth,
            axis x line*=bottom,
            axis y line*=left,
            xlabel={Iteration comparison number},
            ylabel={Relative error},
            ymode=log,
            xmode=normal,
            legend pos = north east,
            ]
            \addplot [fb] table[x=iter,y=relative]{fb_mri_final_noise50_21m.txt};
            \addlegendentry{FB}

            \addplot [projgrad] table[x=iter,y=relative]{fbmg_mri_final_noise50_21m.txt};
            \addlegendentry{FBMG}
        \end{axis}
    \end{tikzpicture}
    \end{subfigure}
    \begin{subfigure}{0.48\columnwidth}
    \begin{tikzpicture}
        \begin{axis}[%
            width = \linewidth,
            height = 0.6\linewidth,
            axis x line*=bottom,
            axis y line*=left,
            xlabel={CPU time},
            ylabel={Relative error},
            ymode=log,
            xmode=normal,
            legend pos = north east,
            xtick = {0.0, 0.04, 0.08, 0.12},
            xticklabel style={
                /pgf/number format/precision=3,
                /pgf/number format/fixed,
            }
            ]

            \addplot [fb] table[x=cputime,y=relative]{fb_mri_final_noise50_21m.txt};
            \addlegendentry{FB}

            \addplot [projgrad] table[x=cputime,y=relative]{fbmg_mri_final_noise50_21m.txt};
            \addlegendentry{FBMG}
        \end{axis}
    \end{tikzpicture}
    \end{subfigure}
    \caption{Relative error \eqref{eq:relerror} versus both iteration count and CPU time for MRI.}
    \label{fig:mri:graphs}
\end{figure}

\begin{figure}[t]
    \centering
    \begin{subfigure}{.195\textwidth}
        \centering
        \includegraphics[width=0.95\linewidth]{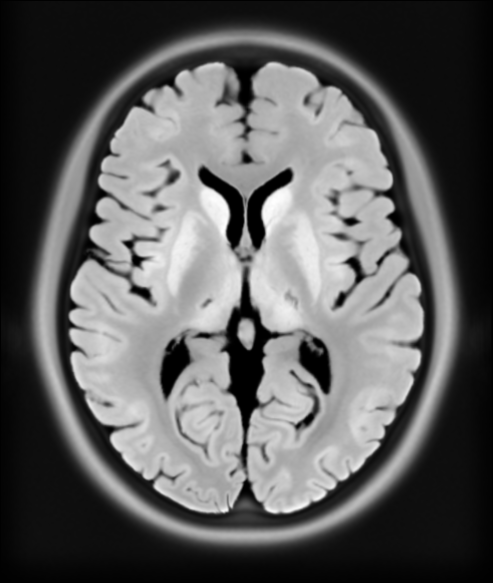}
        \caption{Original}
        \label{fig:original:image:mri}
    \end{subfigure}%
    \begin{subfigure}{.195\textwidth}
        \centering
        \includegraphics[width=0.95\linewidth]{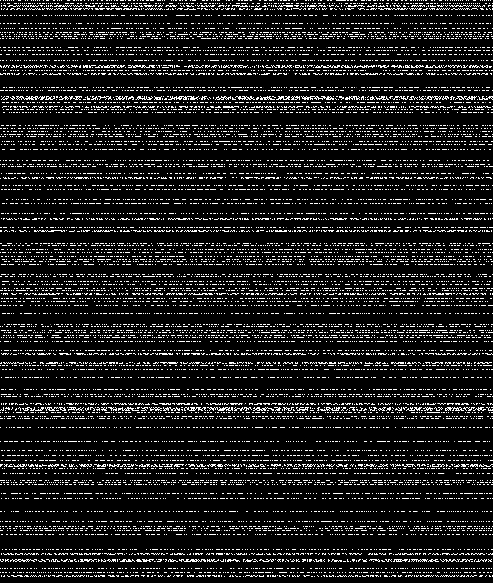}
        \caption{Sample}
        \label{fig:noisy:line:sample:mri}
    \end{subfigure}%
    \begin{subfigure}{.195\textwidth}
        \centering
        \includegraphics[width=0.95\linewidth]{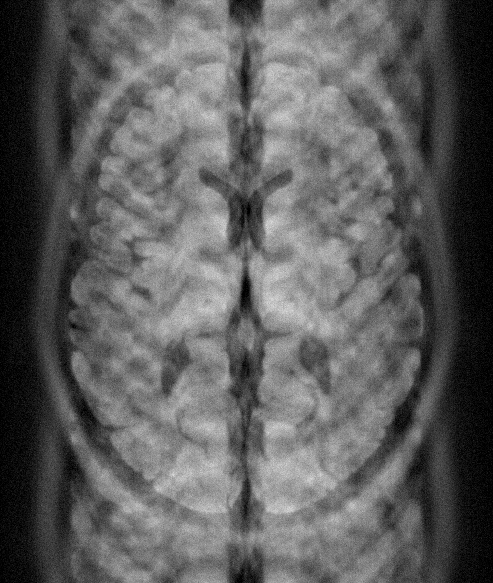}
        \caption{Backproj.}
        \label{fig:noisy:transform:sample:mri}
    \end{subfigure}%
    \begin{subfigure}{.195\textwidth}
        \centering
        \includegraphics[width=0.95\linewidth]{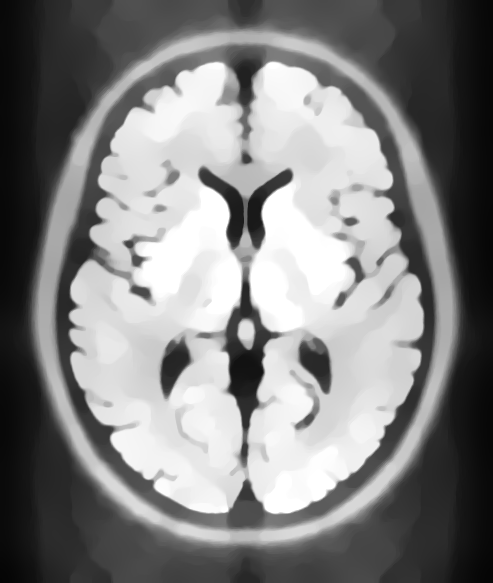}
        \caption{FB}
        \label{fig:fb:solution:image:mri}
    \end{subfigure}%
    \begin{subfigure}{.195\textwidth}
        \centering
        \includegraphics[width=0.95\linewidth]{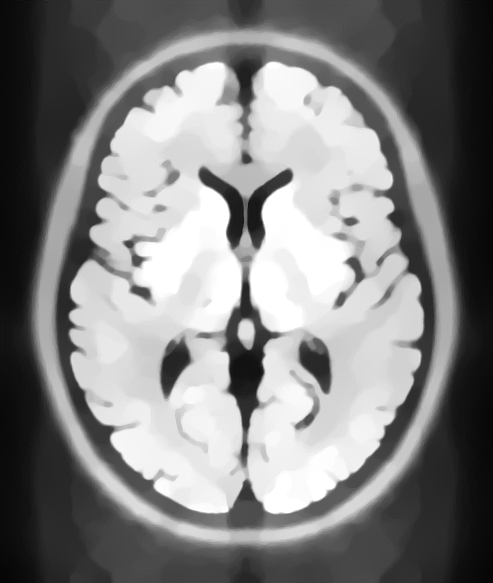}
        \caption{FBMG}
        \label{fig:fbmg:solution:image:mri}
    \end{subfigure}%
    \caption{MRI data and results at relative error $\relerr=0.01$.
    (\subref{fig:noisy:transform:sample:mri}) is the backprojection of the Fourier line sample (\subref{fig:noisy:line:sample:mri}).
    There are altogether $t=100$ such samples.
    }
    \label{fig:mri:compare:FB:FBMG:mri}
\end{figure}

\subsection{Conclusions}

\Cref{fig:denoising:graphs,fig:mri:graphs,tab:error-time} indicate that while the performance improvements in denoising are noticeable, they are \emph{very significant} for the much more expensive MRI problem.
This can be expected, as the fine grid Fourier transform is an expensive operation.
The situation is comparable to \cite{parpas2017multilevel}, who do deblurring directly with the primal problem.
This requires proximal map of total variation to be solved numerically (as a denoising problem) on each fine-grid step, while in the coarse grid they avoid this by using a smooth problem and gradient steps.
For multigrid optimisation methods to be meaningful, it therefore appears that the coarse-grid problems have to be significantly cheaper than the fine-grid problems.

\bibliographystyle{jnsao}

\end{document}